\documentclass[a4paper,11pt,twoside]{amsart}
\usepackage[english]{babel}
\usepackage[utf8]{inputenc}

\usepackage[a4paper,inner=3cm,outer=3cm,top=4cm,bottom=4cm,pdftex]{geometry}
\usepackage{color}
\usepackage{bold-extra}
\usepackage{ mathrsfs }
\usepackage{comment}
\usepackage{graphics}
\usepackage{aliascnt}
\usepackage{enumerate}
\usepackage[pdftex,citecolor=green,linkcolor=red]{hyperref}

\usepackage{amsmath}
\usepackage{amsfonts}
\usepackage{amssymb}
\usepackage{amsthm}
\usepackage{comment}
\usepackage{mathtools}

\newtheorem{theorem}{Theorem}[section]
\newtheorem{corollary}[theorem]{Corollary}

\newtheorem{lemma}[theorem]{Lemma}
\newtheorem{proposition}[theorem]{Proposition}

\theoremstyle{definition}
\newtheorem{definition}[theorem]{Definition}
\newtheorem{remark}[theorem]{Remark}

\newtheorem{question}[theorem]{Question}
\numberwithin{equation}{section}

\renewcommand{\Re}{\textnormal{Re}}
\renewcommand{\Im}{\textnormal{Im}}
\renewcommand\d{\textnormal{d}}

\begin{document}

\title[Invariable generation of almost all alternating groups]{Almost all alternating groups are invariably generated by two elements of prime order}

\author[Joni Ter\"{a}v\"{a}inen]{Joni Ter\"{a}v\"{a}inen}
\address{Department of Mathematics and Statistics, University of Turku, 20014 Turku, Finland}
\email{joni.p.teravainen@gmail.com}

\begin{abstract} We show that for all $n\leq X$ apart from $O(X\exp(-c(\log X)^{1/2}(\log \log X)^{1/2}))$ exceptions, the alternating group $A_n$ is invariably generated by two elements of prime order.  This answers (in a quantitative form) a question of Guralnick, Shareshian and Woodroofe.
\end{abstract}

\maketitle

\section{Introduction}

We say that a finite group $G$ is \emph{invariably generated} by elements $g_1,\ldots, g_k$ if for any $g_1',\ldots, g_k'\in G$ with $g_i'$ belonging to the conjugacy class of $g_i$ we have $\langle g_1',\ldots, g_k'\rangle=G$. In other words, the subset $\{g_1,\ldots, g_k\}$ generates the group even if we replace each element by any of its conjugates.

Invariable generation of finite simple groups has received considerable attention. It is known that every finite simple group is invariably generated by two elements of unspecified order \cite{guralnick-malle}, \cite{lubotzky}. Invariable generation by a few random elements has been studied, among others, in \cite{dixon}, \cite{pemantle-peres-rivin}, \cite{eberhard-ford-green}, \cite{mckemmie}, \cite{garzoni}. The expected number of random elements required for invariable generation has been studied e.g. in \cite{lubotzky}, \cite{lucchini}.  Dolfi, Guralnick, Herzog and Praeger \cite{dghp} asked the question: Which finite simple groups are invariably generated by two elements of prime (or prime power) order? 

We shall focus on invariable generation of the alternating groups $A_n$. For the alternating groups, Shareshian and Woodroofe \cite{shareshian} showed that for $n\geq 8$ a power of two, the group $A_n$ fails to be invariably generated by two elements of prime order. Nevertheless, it is possible that $A_n$ is always generated by an element of prime order together with an element of prime power order; in fact, Guralnick, Shareshian and Woodroofe \cite[Section 5]{guralnick} recently asked this question in the following form (see also \cite[Question 1.2--1.4]{shareshian} and \cite[Section 6]{dghp}).

\begin{question}\label{ques}
For every $n\geq 5$,  is the alternating group $A_n$ invariably generated by an
element whose order is a prime power divisor $p^a$ of $n$, together with an element of prime
order $r>\sqrt{n}$?
\end{question}

In \cite{guralnick}, Guralnick, Shareshian and Woodroofe proved that all $5\leq n\leq 10^{15}$ have this property, which provides considerable numerical evidence for Question \ref{ques}. In \cite{shareshian}, Shareshian and Woodroofe proved that the asymptotic lower density of such $n$ is at least $1-10^{-28}$ (with $a=1$ above).  

In this paper, we prove the following almost-all result on invariable generation of the alternating groups $A_n$. 

\begin{theorem}[Almost all alternating groups are invariably generated by two prime order elements]\label{thm_main} There exists a constant $c>0$ such that, for all $n\leq X$ apart from $\ll X\exp(-c(\log X)^{1/2}(\log \log X)^{1/2})$ exceptions, the alternating group $A_n$ is invariably generated by an element of order $p$ together with an element of order $r$ for some primes $p,r$. Moreover,  we may require that $p\mid n$ and $r>n/\exp(2(\log n)^{1/2}(\log \log n)^{1/2})$.
\end{theorem}

Here, and in the rest of the paper, $c$ stands for a (very small) positive constant that is the same on every occurrence.

It was proved in \cite{shareshian} that under the Riemann hypothesis almost all $n$ satisfy Question~\ref{ques} (with $a=1$). Guralnick, Shareshian and Woodroofe \cite{guralnick} state: ``It would already be somewhat interesting to give a proof that does not rely on the Riemann
Hypothesis that the set of counterexamples to Question~\ref{ques} has asymptotic density $0$.'' Theorem~\ref{thm_almostall} achieves this, with a quantitative ``quasi-polynomial'' saving on the size of the exceptional set.

Theorem \ref{thm_main} will be deduced as a consequence of some group-theoretic considerations combined with the following result on products of exactly two primes in short intervals proved in Section \ref{sec:almostall}.

\begin{theorem}[Power-saving exceptional set for products of two primes in short intervals]\label{thm_almostall}
Let $(\log X)^{C}\leq h\leq X^{1/10}$ for large enough $C\geq 1$. Then,  for all integers $1\leq x\leq X$ apart from $\ll Xh^{-c}$ exceptions, there exist $\geq ch/(\log X)$ products of two primes $p_1p_2\in [x,x+h]$ with $h^{1-2c}\leq p_1\leq h^{1-c}$. 
\end{theorem}

\begin{remark}
The key aspect of Theorem~\ref{thm_almostall} is the size of the exceptional set. By~\cite[Theorem 1.1]{matomaki-tera-primes} (improving on~\cite{tera-primes}), for $h=(\log X)^{2.1}$ the interval $[x,x+h]$  almost always contains products of two primes, and in fact the exceptional set in that result is power-saving in $h$ in the regime $(\log X)^{2.1}\leq h\leq (\log X)^{C}$ (that is, one has an exceptional set of the size $\ll X/h^{c_0}$ for some constant $c_0>0$). However, in the complementary range $h\geq (\log X)^{\psi(X)}$ with $\psi(X)$ tending to infinity relatively rapidly, the method there does not give such a good exceptional set.

It turns out that we will need Theorem \ref{thm_almostall} only for $h=\exp((\log X)^{1/2}(\log \log X)^{1/2})$, but we give a proof in the larger range $(\log X)^{C}\leq h\leq X^{1/10}$ as it may be of independent interest. 

 We also remark that even under the Riemann hypothesis we are not aware of a proof that there are $\ll Xh^{-1/2-\varepsilon}$ exceptional intervals $[x,x+h]$ with $x\leq X$ not containing a product of two primes, with $\varepsilon>0$ fixed. 
\end{remark}

By  \cite[eq. (1.1)]{shareshian}, our main theorem has the following implication for common prime divisors of binomial coefficients.

\begin{corollary}\label{cor_binomial}
There exists a constant $c>0$ such that, for all $n\leq X$ apart from $\ll X\exp(-c(\log X)^{1/2}(\log \log X)^{1/2})$ exceptions, there exist two primes $p_1,p_2$ (depending on $n$) such that for each $1\leq i\leq n-1$ at least one of $p_1,p_2$ divides $\binom{n}{i}$.
\end{corollary}

This improves on \cite[Theorem 1.5]{shareshian}, where it was shown that the set of exceptional $n\leq X$ has size $\leq (10^{-28}+o(1))X$. One would expect that there are no exceptional $n$. Let us also mention that assuming very strong information on primes in short intervals, namely  Cram\'er's conjecture, one can show that there are $O(X^{1/2+o(1)})$ exceptional $n\leq X$ both for Corollary \ref{cor_binomial} and Theorem \ref{thm_main} (see \cite[Subsection 4.3]{guralnick}).

\subsection{Acknowledgments} The author thanks Ben Green for bringing the topic of invariable generation to his attention. The author also thanks the referees for helpful comments. The author was supported by Academy of Finland grant no. 340098. 

\section{Notation}\label{subsec:notation} The symbols $p,p_i,r$ always stand for prime numbers.

We denote by $(a,b)$ the greatest common divisor of two natural numbers $a,b$. As usual, $\Lambda$ denotes the von Mangoldt function.

We use the Vinogradov asymptotic notation $A\ll B$ to denote that there exists a constant $C$ such that $|A|\leq CB$.

In the course of the proof, we shall need maximal subgroups of $A_n$.
The maximal subgroups $H$ of $A_n$ are classified into three types:
\begin{itemize}
    \item We say that $H$ is \emph{intransitive} if there exist $i,j\in [n]$ such that under the natural action of $H$ on $[n]$ we have $i\cdot h\neq j$ for all $h\in H$. Otherwise, we say that $H$ is transitive. If $H$ is an intransitive maximal subgroup, then $H$ fixes some set $X\subset [n]$ with $1\leq |X|<n$ under the action of $H$ on $[n]$. 
    
    \item We say that $H$ is \emph{imprimitive} if it is transitive and there is a proper partition $\pi$ of $[n]$ into parts of size $\geq 2$ such that the action of $H$ on $[n]$ permutes these parts. By the maximality of $H$, we may assume that the parts in $\pi$ all have the same size.
    
    \item We say that $H$ is \emph{primitive} if it is transitive but not imprimitive.
\end{itemize}

It is clear that each maximal subgroup must be of one of these three types. 

\section{Group-theoretic lemmas}

The group theory part of our argument can be abstracted into similar ingredients as the arguments in \cite{guralnick}. 

\begin{proposition}\label{thm_characterize}
Let $n\geq 25$ be an integer. Suppose that the following hold for some primes $r>\sqrt{2n}$ and $p\mid  n$ and for $t=\lfloor n/r\rfloor$.
\begin{enumerate}[(i)]
 
    \item $n$ is not a prime power.
    
    \item $n$ is not of the form $(q^d-1)/(q-1)$ for any integers $q\geq 2$ and $d\geq 3$.
    
    \item $n-tr\geq 3$.
    
    \item $(t,n)=1$.
 
    \item $p\nmid ar+b$ for any integers $a,b$ with $0\leq a\leq t$, $0\leq b\leq n-tr$, and $0<ar+b<n$. 
    
\end{enumerate}
Then $A_n$ is invariably generated by an element of order $p$ together with an element of order $r$.
\end{proposition}

\begin{proof} We may assume that $p\geq 3$, since if $p=2$ condition (v) cannot hold. Denote $u=n-tr\in [3,r)$. Let $g_1\in A_n$ be a product of $n/p$ disjoint cycles of length $p$, and let $g_2\in A_n$ be a product of $t$ disjoint cycles of length $r$ and $u$ fixed points. Since disjoint cycles commute, $g_1$ has order $p$ and $g_2$ has order $r$. Since conjugation preserves cycle structure, any conjugates of $g_1,g_2$ are still of the same form. Hence, it suffices to show that $\langle g_1,g_2\rangle=A_n$. 

Suppose that $\langle g_1,g_2\rangle\neq A_n$. Let $H\neq A_n$ be the maximal subgroup of $A_n$  that contains $\langle g_1,g_2\rangle$. By the classification of maximal subgroups of $A_n$ (Subsection \ref{subsec:notation}), $H$ must be either primitive, imprimitive or intransitive.

\textbf{Case 1.} Suppose $H$ is primitive. Recall that $H$ contains $g_2$, which is a product of $t$ $r$-cycles and $u$ fixed points. By \cite[Theorem 2.4 and Remark 2.5]{guralnick}, this implies that one of the following holds:
\begin{itemize}

\item $n=\frac{q^d-1}{q-1}$ for some integers $q\geq 2$ and $d\geq 3$;

\item $u\leq 2$;

\item $n$ is a prime power.
\end{itemize}
However, these are all impossible by our assumptions (i), (ii), (iii).

\textbf{Case 2.} Suppose $H$ is imprimitive. Then $H$ preserves some partition $\pi$ of $[n]$ into $d$ parts of size $n/d$ for some $d\mid n$, $1<d<n$. Note that the base $r$ representation of $n$ is $n=tr+u$. Hence, $g_2$ is a \emph{base $r$-element}\footnote{We say that an element $g\in A_n$ is a \emph{base-$p$ element} if, given the base $p$ representation $n=\alpha_0+\alpha_1p+\alpha_2p^2+\cdots$, the element $g$ has $\alpha_0$ fixed points and $\alpha_i$ cycles of length $p^i$ for all $i$.}. By \cite[Lemma 3.6]{guralnick}, the base $r$-element $g_2$ fixing $\pi$ implies that $d\mid(t,u)$ or $n/d\mid (t,u)$. In either case, $(t,u)>1$, so that also $(t,n)=(t,tr+u)>1$, which contradicts our assumption (iv).

\textbf{Case 3.} Lastly, suppose $H$ is intransitive. Then there is a subset $X\subset [n]$ of some size $1<k<n$ such that $H$ is the stabilizer of $X$. Now, $g_1$ fixing $X$ implies that $k=pm$ for some integer $m$, while $g_2$ fixing $X$ implies that $k=ar+b$ for some $0\leq a\leq t$, $0\leq b\leq u$, with $0<ar+b<n$. But by assumption (v) both of these cannot happen.
\end{proof}

Note that Proposition \ref{thm_characterize} does not handle the case of prime (or prime power) $n$; for these we need the following complementary lemma.

\begin{lemma}\label{le_primecase} Let $n\geq 2$. Let $p,r$ be primes with $p\mid n$ and  $r<n-2<n\leq r+p$. Then $A_n$ is invariably generated by an element of order $r$ together with an element of order $p$.

\end{lemma}
\begin{proof}
This follows from \cite[Lemma 3.3]{guralnick} with $a=1$ by taking the permutation $x$ there to be a product of $n/p$ disjoint cycles of length $p$.
\end{proof}

\section{Proof of Theorem \ref{thm_main} assuming Theorem \ref{thm_almostall}}

In this section, we prove  Theorem \ref{thm_main} assuming Theorem \ref{thm_almostall} (which in turn is proved in Section~\ref{sec:almostall}). Throughout this section, let
\begin{align*}
 h:=\exp((\log X)^{1/2}(\log \log X)^{1/2}).
 \end{align*}
By Theorem \ref{thm_almostall}, for all $n\leq X$ outside an admissible exceptional set, there exist $\gg h/(\log X)$ products of two primes 
\begin{align*}
p_1p_2\in [n-h,n-3] \textnormal{ with } h^{1-2c}/2\leq p_1\leq h^{1-c}.    
\end{align*}
By Proposition \ref{thm_characterize} (with $r=p_2$, $t=p_1$) and the union bound, it suffices to show that each of the assumptions (i)--(v) of Proposition \ref{thm_characterize} fails for $\ll X\exp(-c(\log X)^{1/2}(\log \log X)^{1/2})$ integers $n\leq X$. Assumption (iii) is automatically satisfied with our choices. If (i) fails, then $n=p^{a}$ is a prime power, and if $a=1$ then Lemma~\ref{le_primecase} tells us that $A_n$ is generated by an element of order $r$ together with an element of order $p$. There are $\ll X^{1/2}$ integers $n\leq X$ of the form $p^{a}$ with $p$ prime and $a\geq 2$, so this is also an acceptable exceptional set. We are left with showing that the properties (ii), (iv), (v) are true for all but the stated number of exceptional $n$. The smallness of exceptions to assumptions (ii), (iv), (v) will follow from the following five lemmas.

\begin{lemma}[Dealing with very large prime factors]\label{le0}
Suppose that $n$ is large enough and that $p\mid n$ for some prime $p\geq n^{0.9}$. Then, $A_n$ is invariably generated by an element of order $p$ together with an element of prime order $r>n/2$.
\end{lemma}

\begin{proof}
By the prime number theorem in short intervals\footnote{One could replace the exponent $0.9$ here by $0.525$ by \cite{baker-harman-pintz}, but the above suffices for our purposes.}, there is a prime $r\in (n-n^{0.9},n-3]\subset (n-p,n-3]$ for all $n\geq N_0$. Hence, the claim follows from Lemma \ref{le_primecase}.
\end{proof}

\begin{lemma}[Exceptions to assumption (ii)]\label{le1}
The number of $n\leq X$ that are of the form $(q^d-1)/(q-1)$ with $q\geq 2$ and $d\geq 3$ is $\ll \sqrt{X}$.
\end{lemma}

\begin{proof}
The number of $n\leq X$ of the form $(q^3-1)/(q-1)=1+q+q^2$ is trivially $\ll \sqrt{X}$. Similarly, for any $d\geq 4$, the number of $n$ of the form $(q^d-1)/(q-1)$ is $\ll X^{1/(d-1)}\ll X^{1/3}$. Since necessarily $d\leq (\log X)/(\log 2)$, the claim follows.
\end{proof}

\begin{lemma}[Exceptions to assumption (iv)]\label{le2}
Let $n\in [X^{1/2},X]$. The number of products of two primes $p_1p_2\in [n-h,n]$ with $h^{1-2c}/2\leq p_1\leq h^{1-c}$ and $p_1\mid n$ is $o(h/(\log X))$. 
\end{lemma}

\begin{proof}
We can very crudely bound
\[\pushQED{\qed} 
\sum_{\substack{h^{1-2c}/2\leq p_1\leq h^{1-c}\\p_1\mid n}}\,\,\sum_{(n-h)/p_1\leq p_2\leq n/p_1}1\ll \sum_{\substack{p_1\geq h^{1-2c}/2\\p_1\mid n}}\frac{h}{p_1}\ll \frac{h(\log X)}{h^{1-2c}}=o\left(\frac{h}{\log X}\right).
\qedhere  
\]
\end{proof}

\begin{lemma}[Exceptions to assumption (v) with large prime divisor]\label{le3}
For all but $\ll X/h^{1/2}$ integers $n\leq X$ the following holds. 

The number of products of two primes $p_1p_2\in [n-h,n]$ with $h^{1-2c}/2\leq p_1\leq h^{1-c}$ that satisfy $p\mid ap_2+b$ for some prime $p\mid n$, $h^3\leq p\leq X^{0.9}$ and some $0\leq a\leq p_1$, $0\leq b\leq h$ with $0<ap_2+b<n$ is $o(h/(\log X))$. 
\end{lemma}

\begin{proof}
Suppose $p\mid ap_2+b$ with $0\leq a\leq p_1$ and $0\leq b\leq h$. If $a=p_1$, then $p\mid n$, $p\mid ap_2+b$ implies $p\mid n-p_1p_2-b\in (0,h]$. But as $p>h$, this is not possible. Similarly, if $a=0$, then $p\mid b\in [1,h]$, which contradicts $p>h$. Now, denoting
\begin{align*}
\mathcal{S}_p:=\{j\in \mathbb{Z}:\,\, aj+b\equiv 0\pmod p\textnormal{ for some } 1\leq a<p_1,0\leq b\leq h\},    
\end{align*}
we can write the condition $p\mid ap_2+b$ with $a,b$ as above in the form $p_2\in \mathcal{S}_p$. It then suffices to show, for every $h^3\leq p\leq X^{0.9}$, that
\begin{align*}
\sum_{\substack{n-h\leq p_1p_2\leq n\\p_2\in \mathcal{S}_p}}1\leq \frac{h}{(\log X)^2}    
\end{align*}
holds for all but $\ll X/(h^{0.9}p)$ integers $n\leq X$, $n\equiv 0\pmod p$. Indeed, once we have this, the claim follows from the union bound and the fact that any $n\leq X$ has $\ll (\log X)/(\log h)$ prime factors $p>h^3$.

Using the inequality $|\{n\leq X:b_n\geq \lambda\}|\leq \lambda^{-1}\sum_{n\leq X}b_n$ with $b_n\geq 0$, we can estimate
\begin{align}\label{eq6}
&\Big|\Big\{n\leq X:\,\, p\mid n\,\textnormal{ and }  \sum_{\substack{n-h\leq p_1p_2\leq n\\h^{1-2c}/2\leq p_1\leq h^{1-c}\\p_2\in \mathcal{S}_p}}1\geq \frac{h}{(\log X)^2}\Big\}\Big|\nonumber\\
&\leq \frac{(\log X)^2}{h}\sum_{m\leq X/p}\,\,\sum_{h^{1-2c}/2\leq p_1\leq h^{1-c}}\,\,\sum_{\substack{(pm-h)/p_1\leq \ell \leq pm/p_1\\\ell\in \mathcal{S}_p}}1.
\end{align}
Note then that $\ell\in \mathcal{S}_p$ for some $\ell\in [(pm-h)/p_1,pm/p_1]$ implies that for some $1\leq a<p_1$ we have
\begin{align*}
\frac{apm}{p_1}\in [-2h,2h]\pmod p.    
\end{align*}
Therefore, we have
\begin{align*}
\frac{am}{p_1}\in \left[-\frac{2h}{p},\frac{2h}{p}\right]\pmod 1.    
\end{align*}
But if $p_1\nmid am$, then by denoting by $\|\cdot\|$ the distance to the nearest integer, we have
\begin{align*}
\left\|\frac{am}{p_1}\right\|\geq \frac{1}{p_1}>\frac{2h}{p},    
\end{align*}
since $p\geq h^3$ and $p_1\leq h^{1-c}$. We must therefore have $p_1\mid am$, so $p_1\mid m$. Hence, \eqref{eq6} is bounded by
\begin{align*}
\\
&\ll \frac{(\log X)^2}{h}\sum_{m\leq X/p}\,\,\sum_{h^{1-2c}/2\leq p_1\leq h^{1-c}}\frac{h}{p_1}1_{p_1\mid m}\\
&\ll (\log X)^2\sum_{h^{1-2c}/2\leq p_1\leq h^{1-c}}\frac{X}{pp_1^2}\\
&\ll \frac{X}{ph^{0.9}},
\end{align*}
recalling that $p_1p\ll X^{0.9+o(1)}$ and $p_1\geq h^{1-2c}\geq h^{0.9}(\log X)^2$ if we take $c<1/25$. As noted before, this was enough to conclude the proof.
\end{proof}

\begin{lemma}[Bounding the number of smooth numbers]\label{le4}
The number of $n\leq X$ that have no prime factors larger than $h^3$ is $\ll X\exp(-c(\log X)^{1/2}(\log \log X))$.
\end{lemma}

\begin{proof}
Let $s=(\log X)/(\log h^3)=(\log X)^{1/2}/(3(\log \log X)^{1/2})$. Then, by a standard smooth number estimate (see e.g. \cite[Corollary 1.3]{ht}), the number of $h^3$-smooth integers up to $X$ is
\begin{align*}
\ll  Xs^{-(1+o(1))s}\ll X\exp(-c(\log X)^{1/2}(\log \log X)^{1/2}),    
\end{align*}
provided we take $c< 1/6$.
\end{proof}

Combining Lemmas \ref{le0}, \ref{le1}, \ref{le2}, \ref{le3} and \ref{le4}, Theorem \ref{thm_main} follows (assuming still Theorem~\ref{thm_almostall}). 

\begin{remark}
Note that the size of our exceptional set arose essentially from solving the equation $h^{-c}=s^{-s}$ with $s=(\log X)/(\log h)$ and $c\asymp 1$. Since it does not seem easy to obtain a saving larger than $h^{-O(1)}$ for the size of the exceptional set in Theorem \ref{thm_almostall} even under the Riemann hypothesis, it seems that a new idea would be required to improve on the size of our exceptional set in Theorem~\ref{thm_main}.
\end{remark}

\section{Proof of Theorem \ref{thm_almostall}}\label{sec:almostall}

Throughout this section, let $\varepsilon>0$ be a small enough absolute constant. We can restate Theorem \ref{thm_almostall} in the following quantitative form.

\begin{theorem}\label{thm_almostall2} Let $(\log x)^{C}\leq h\leq X^{1/10}$ for large enough $C\geq 1$. Then for all integers $1\leq x\leq X$ apart from $\ll Xh^{-\varepsilon}$ exceptions we have
\begin{align}\label{eq1}
\sum_{\substack{x\leq n_1n_2\leq x+h\\h^{1-2\varepsilon}\leq n_1\leq h^{1-\varepsilon}}}\Lambda(n_1)\Lambda(n_2)=h\sum_{h^{1-2\varepsilon}\leq n_1\leq h^{1-\varepsilon}}\frac{\Lambda(n_1)}{n_1}+O(h(\log X)^{-100}).  
\end{align}
\end{theorem}

By Mertens's theorem, we have
\begin{align*}
 \sum_{h^{1-2\varepsilon}\leq n_1\leq h^{1-\varepsilon}}\frac{\Lambda(n_1)}{n_1}=\left(\log \frac{1-\varepsilon}{1-2\varepsilon}+o(1)\right)\log h  
\end{align*}
and $\log((1-\varepsilon)/(1-2\varepsilon))>\varepsilon$ for $\varepsilon\in (0,1/2)$. Hence, noting that $\Lambda(n_1)\leq \log h, \Lambda(n_2)\leq \log X+1$ and trivially bounding the contribution of the higher prime powers to $\Lambda$, we see that Theorem \ref{thm_almostall2} directly implies Theorem \ref{thm_almostall} with $c=\varepsilon$.

We will prove Theorem \ref{thm_almostall2} via the method of Dirichlet polynomials.
The main hurdle in the proof is that we do not know a zero-free strip of constant width for the Riemann zeta function. Given our current knowledge on the zero-free region of the Riemann zeta function (i.e., the Vinogradov--Korobov zero-free region), we cannot hope to have an error term better than $h\exp(-(\log X)^{1/3+o(1)})$ on the right of \eqref{eq1}. Hence, \eqref{eq1} cannot be directly converted into a variance estimate that we could hope to unconditionally prove. This issue is amended by defining a \emph{model function} $\widetilde{\Lambda}$  for the von Mangoldt function such that $\widetilde{\Lambda}$  ``resonates'' with the zeros of the Riemann zeta function of large real part in exactly the same way as the von Mangoldt function itself, and therefore the Dirichlet polynomial of $\Lambda-\widetilde{\Lambda}$ satisfies power-saving Dirichlet polynomial bounds. More precisely, we define $\widetilde{\Lambda}$ as follows.

\begin{definition}[A model for the von Mangoldt function] For a given $X\geq 2$, define 
\begin{align*}
\widetilde{\Lambda}(n):=1-\sum_{\substack{\rho=\beta+i\gamma\\\beta\geq 1-10\varepsilon\\|\gamma|\leq X^{1.1}}}n^{\rho-1},    
\end{align*}
where the sum is over the nontrivial zeros $\rho$ of the Riemann zeta function.  
\end{definition}

We have the following lemma on the size of the model function $\widetilde{\Lambda}(n)$.

\begin{lemma}\label{le_bound}
For $n\in [X^{0.1},2X]$, we have $|\widetilde{\Lambda}(n)-1|\ll \exp(-(\log X)^{0.33})$.
\end{lemma}

\begin{proof}
We have
\begin{align*}
\widetilde{\Lambda}(n)-1&=-\sum_{\substack{\rho=\beta+i\gamma\\\beta\geq 1-10\varepsilon\\|\gamma|\leq X^{1.1}}}n^{\rho-1}\\
&\ll \sum_{\substack{\rho=\beta+i\gamma\\\beta\geq 1-10\varepsilon\\|\gamma|\leq X^{1.1}}}X^{0.1(\beta-1)}.
\end{align*}
By the Vinogradov--Korobov zero-free region, we necessarily have $\beta\leq \beta_0:=1-(\log X)^{-0.667}$ for $X\geq X_0$. Hence, by splitting the values of $\beta$ into intervals of length $\leq 1/(\log X)$, we have
\begin{align*}
\sum_{\substack{\rho=\beta+i\gamma\\\beta\geq 1-10\varepsilon\\|\gamma|\leq X^{1.1}}}X^{0.1(\beta-1)}\ll (\log X) \max_{1-10\varepsilon\leq \beta\leq \beta_0}  X^{0.1(\beta-1)}N(\beta,X^{1.1}),  
\end{align*}
where $N(\beta,T)$ denotes the number of zeros $\rho$ of the Riemann zeta function with $\Re(\rho)\geq \beta$, $|\Im(\rho)|\leq X$. We may assume that $\varepsilon\leq 10^{-8}$, say. By a zero density estimate for the Riemann zeta function near the $1$-line  \cite{ford}, for $\beta\geq 1-2\cdot 10^{-7}$ (say) we have
\begin{align}\label{eq7}
N(\beta,X^{1.1})\ll (X^{1.1})^{100(1-\beta)^{3/2}}\ll X^{0.1(1-\beta)/2}.    
\end{align}
Hence we have 
\begin{align*}
\max_{1-10\varepsilon\leq \beta\leq \beta_0}  (\log X)X^{0.1(\beta-1)}N(\beta,X^{1.1})\ll X^{0.04(\beta_0-1)}\ll \exp(-(\log X)^{0.33}), 
\end{align*}
giving the claim.
\end{proof}

With the help of our model function, we can state a variance estimate that will turn out to imply Theorem \ref{thm_almostall2}.

\begin{proposition}[A variance estimate]\label{prop_variance} Let $(\log X)^{C}\leq h\leq X^{1/9}$ for large enough $C\geq 1$. Also let $H=X^{1-10\varepsilon}$. Then we have
\begin{align*}
\int_{X/2}^{X}\left|\sum_{\substack{x\leq n_1n_2\leq x+h\\h^{1-2\varepsilon}\leq n_1\leq h^{1-\varepsilon}}}\Lambda(n_1)(\Lambda-\widetilde{\Lambda})(n_2)-\frac{h}{H}\sum_{\substack{x\leq n_1n_2\leq x+H\\h^{1-2\varepsilon}\leq n_1\leq h^{1-\varepsilon}}}\Lambda(n_1)(\Lambda-\widetilde{\Lambda})(n_2)\right|^2\d x\ll h^{2-4\varepsilon}X.  
\end{align*}
\end{proposition}

\begin{proof}[Proof of Theorem \ref{thm_almostall2} assuming Proposition \ref{prop_variance}] By Lemma \ref{le_bound} and  Chebyshev's inequality, it suffices to show for all $x\in [X/2,X]$ that
\begin{align}\label{eq2}
\sum_{\substack{x\leq n_1n_2\leq x+H\\h^{1-2\varepsilon}\leq n_1\leq h^{1-\varepsilon}}}\Lambda(n_1)(\Lambda(n_2)-1)\ll_A H(\log X)^{-A}   
\end{align}
and
\begin{align}\label{eq3}
\sum_{\substack{x\leq n_1n_2\leq x+h'\\h^{1-2\varepsilon}\leq n_1\leq h^{1-\varepsilon}}}\Lambda(n_1)(\widetilde{\Lambda}(n_2)-1)\ll_A h'(\log X)^{-A}    
\end{align}
for $h'\in \{h,H\}$. 

The first claim \eqref{eq2} follows directly by writing 
\begin{align*}
\sum_{\substack{x\leq n_1n_2\leq x+H\\h^{1-2\varepsilon}\leq n_1\leq h^{1-\varepsilon}}}\Lambda(n_1)(\Lambda(n_2)-1)=\sum_{h^{1-2\varepsilon}\leq n_1\leq h^{1-\varepsilon}}\Lambda(n_1)\sum_{x/n_1\leq n_2\leq (x+H)/n_1}(\Lambda(n_2)-1)    
\end{align*}
and applying the prime number theorem in short intervals to the $n_2$ sum.

For the proof of \eqref{eq3}, note simply that by Lemma \ref{le_bound} for $x\in [X/2,X]$ we have
\[\pushQED{\qed} 
\sum_{\substack{x\leq n_1n_2\leq x+h'\\h^{1-2\varepsilon}\leq n_1\leq h^{1-\varepsilon}}}\Lambda(n_1)|\widetilde{\Lambda}(n_2)-1|\ll \sum_{h^{1-2\varepsilon}\leq n_1\leq h^{1-\varepsilon}} \frac{(\log h) h'}{n_1}\exp(-(\log X)^{0.33})\ll_A \frac{h'}{(\log X)^{A}}.\qedhere  
\]
\end{proof}

Before proving Proposition \ref{prop_variance}, we shall reduce it to mean squares of Dirichlet polynomials.

\begin{proposition}[Mean square bound for a product of two prime Dirichlet polynomials]\label{prop_dirichlet} Let
\begin{align*}
 P_1(s):=\sum_{h^{1-2\varepsilon}\leq n\leq h^{1-\varepsilon}}\Lambda(n)n^{-s},\quad \widetilde{P}(s)=\sum_{X/(2h^{1-\varepsilon})\leq n\leq X/h^{1-2\varepsilon}}(\Lambda-\widetilde{\Lambda})(n)n^{-s}.
\end{align*}
Also let $(\log X)^{C}\leq h\leq X^{1/9}$ with $C\geq 1$ large enough. Then we have
\begin{align*}
\int_{h^{10\varepsilon}}^{X/h^{1-4\varepsilon}}|P_1(1+it)|^2|\widetilde{P}(1+it)|^2\d t\ll h^{-4\varepsilon}. \end{align*}
\end{proposition}

\begin{proof}[Proof of Proposition \ref{prop_variance} assuming Proposition \ref{prop_dirichlet}] This is a standard Perron formula argument. Let
\begin{align*}
a_n=\sum_{\substack{n=n_1n_2\\h^{1-2\varepsilon}\leq n_1\leq h^{1-\varepsilon}\\X/(2h^{1-\varepsilon})\leq n_2\leq X/h^{1-2\varepsilon}}}\Lambda(n_1)\Lambda(n_2),\quad S_y(x)=\frac{1}{y}\sum_{x\leq n\leq x+y}a_n,\quad F(s)=\sum_n a_nn^{-s}.
\end{align*}
Also let $H=X^{1-10\varepsilon}$. By \cite[Lemma 1]{tera-primes}, we have\footnote{In \cite[Lemma 1]{tera-primes}, $a_n$ is assumed to be supported in $[X,2X]$, but this is actually not used in the proof.}
\begin{align*}
\int_{X/2}^{X}\left|\frac{1}{h}S_h(x)-\frac{1}{H}S_H(x)\right|^2\d x\ll h^{-10\varepsilon}+\int_{h^{10\varepsilon}}^{X/h}|F(1+it)|^2\d t+\max_{T\geq X/h}\frac{X}{Th}\int_{T}^{2T}|F(1+it)|^2\d t.   \end{align*}
Now the claim follows by applying the mean value theorem for Dirichlet polynomials in the range $T\geq X/h^{1-4\varepsilon}$.
\end{proof}

Before proving Proposition \ref{prop_dirichlet}, we need one more lemma.

\begin{lemma}\label{le_pointwise}
For $2\leq |t|\leq X$, we have
\begin{align*}
|\widetilde{P}(1+it)|\ll 1/|t|+X^{-8\varepsilon}.    
\end{align*}
\end{lemma}

\begin{proof}
Let $Q_1=X/(2h^{1-\varepsilon})$, $Q_2=X/h^{1-2\varepsilon}$. By a slight variant of the explicit formula (which is proved in the same way; cf. arguments in \cite[Section 5]{iw-kow}), for $2\leq |t|\leq X$ we have
\begin{align*}
\sum_{Q_1\leq n\leq Q_2}\Lambda(n)n^{-1-it}=-\frac{Q_2^{it}-Q_1^{it}}{it}-\sum_{\substack{\rho=\beta+i\gamma\\|\gamma|\leq X^{1.1}}} \frac{Q_2^{\rho-1-it}-Q_1^{\rho-1-it}}{\rho-it}+O\left(\frac{Q_2(\log Q_2)^3}{X^{1.1}}\right).    
\end{align*}
Here the error term is $\ll X^{-10\varepsilon}$ (if $\varepsilon\leq 1/110$) and the first term is $\ll |t|^{-1}$. Note that the contribution of $\beta<1-10\varepsilon$ above is $\ll X^{-8\varepsilon}$, using $\sum_{\rho:|\Im(\rho)|\leq X^{1.1}}1/|\rho-it|\ll (\log X)^2$ and $Q_1\gg X/h\gg X^{8/9}$.

On the other hand, we have
\begin{align*}
\sum_{Q_1\leq n\leq Q_2}\widetilde{\Lambda}(n)n^{-1-it}=\sum_{Q_1\leq n\leq Q_2}n^{-1-it}- \sum_{\substack{\rho=\beta+i\gamma\\\beta\geq 1-10\varepsilon\\|\gamma|\leq X^{1.1}}}\sum_{Q_1\leq n\leq Q_2}n^{\rho-2-it}.   
\end{align*}
By Perron's formula, for any $\xi$ with $\Re(\xi)\leq 1$ we have
\begin{align*}
\sum_{Q_1\leq n\leq Q_2}n^{\xi-2-it}&=\frac{1}{2\pi i}\int_{1-X^{10\varepsilon}i}^{1+X^{10\varepsilon}i}\zeta(s+2-\xi+it)\frac{Q_2^s-Q_1^s}{s}\d s+O(X^{-9\varepsilon}).
\end{align*}
Shifting the line of integration to $\Re(s)=\Re(\xi)-2$ and applying 
the residue theorem and the estimate $|\zeta(iu)|\ll (1+|u|)^{1/2}$, this is
\begin{align*}
\frac{Q_2^{\xi-1-it}-Q_1^{\xi-1-it}}{\xi-1-it}+O(X^{-9\varepsilon}),    
\end{align*}
since either the simple pole at $s=\xi-1-it$ is captured by the integral unless $|\Im(\xi)-t|\geq X^{10\varepsilon}$.
Applying the above with $\xi=1$ and $\xi=\rho$, and recalling \eqref{eq7}, the claim follows.
\end{proof}

\begin{proof}[Proof of Proposition \ref{prop_dirichlet}]  We apply the Matom\"aki--Radziwi\l{}\l{} method \cite{mr-annals}. We split the integration domain into two sets
\begin{align*}
\mathcal{T}_1=\{h^{10\varepsilon}\leq t\leq X/h^{1-4\varepsilon}:\,\, |P_1(1+it)|\geq h^{-5\varepsilon}\},\quad \mathcal{T}_2=[h^{10\varepsilon},X/h^{1-4\varepsilon}]\setminus \mathcal{T}_1.    
\end{align*}
By the mean value theorem for Dirichlet polynomials, we trivially have
\begin{align*}
\int_{\mathcal{T}_1}|P_1(1+it)|^2|\widetilde{P}(1+it)|^2\d t\ll h^{-10\varepsilon}\left(\frac{X/h^{1-4\varepsilon}+X/h^{1-\varepsilon}}{X/h^{1-\varepsilon}}\right)\ll h^{-4\varepsilon}.    
\end{align*}
Consider then the integral over $\mathcal{T}_2$. By a large values estimate (\cite[Lemma 6]{tera-primes},  which is proved by raising $P_1$ to a large power and applying the mean value theorem), if $\mathcal{U}\subset \mathcal{T}_2$ is any well-spaced subset (i.e., any two of its elements are separated by $\geq 1$), then (taking $C$ large in terms of $\varepsilon$) we have
\begin{align*}
|\mathcal{U}|\ll X^{11\varepsilon}.    
\end{align*}
On the other hand, for some well-spaced $\mathcal{U}\subset \mathcal{T}_2$ we have
\begin{align*}
 \int_{\mathcal{T}_2}|P_1(1+it)|^2|\widetilde{P}(1+it)|^2\d t\ll   \sum_{t\in \mathcal{U}}|P_1(1+it)|^2|\widetilde{P}(1+it)|^2,
\end{align*}
and crudely bounding $|P_1(1+it)|\ll 1$ and using Lemma \ref{le_pointwise}, we can bound this by
\begin{align*}
\ll h^{-10\varepsilon}+|\mathcal{U}|X^{-16\varepsilon}\ll h^{-10\varepsilon}+X^{-5\varepsilon}\ll h^{-4\varepsilon}.    
\end{align*}
This proves the claim.
\end{proof}

\bibliography{refs}
\bibliographystyle{plain}

\end{document}